\documentclass[12pt,leqno]{article}

\usepackage{amsfonts,mathrsfs}
\usepackage{amssymb,amsmath,amsbsy,amsthm}
\usepackage{blkarray}
\usepackage{tikz}
\usepackage{tkz-euclide}
\usepackage{tikz-cd}
\usepackage{ae}
\usepackage{bm}
\usepackage{graphicx}
\usepackage{float}
\usepackage{cite}
 \usepackage{indentfirst}
\usepackage{lineno}
\usepackage{titlesec}
\usepackage{enumitem}
\usepackage{color}
\usepackage{aliascnt}
\usepackage{varioref}
\usepackage{hyperref}
\hypersetup{colorlinks=true,linkcolor=cyan,anchorcolor=cyan,citecolor=red,CJKbookmarks=True,
,urlcolor=cyan}
\usepackage{cleveref}

\numberwithin{equation}{section}

\def\And{\mbox{\rm ~and~}}

\def\i{\mbox{\rm (\hspace{0.2mm}i\hspace{0.2mm})}\,}

\def\sst{\scriptscriptstyle}

\def\({\mbox{\rm (}}\def\){\mbox{\rm )}}

\def\b{\big}

\makeatletter

\newcommand{\Rmnum}[1]{\expandafter\@slowromancap\romannumeral #1@}
\makeatother
\newtheorem{theorem}{Theorem}[section]
\newaliascnt{lemma}{theorem}
\newtheorem{lemma}[lemma]{Lemma}
\aliascntresetthe{lemma}

\newaliascnt{proposition}{theorem}
\newtheorem{proposition}[proposition]{Proposition}
\aliascntresetthe{proposition}

\newaliascnt{fact}{theorem}
\newtheorem{fact}[fact]{Fact}
\aliascntresetthe{fact}

\newaliascnt{definition}{theorem}
\newtheorem{definition}[definition]{Definition}
\aliascntresetthe{definition}

\newaliascnt{conjecture}{theorem}
\newtheorem{conjecture}[conjecture]{Conjecture}
\aliascntresetthe{conjecture}

\newaliascnt{corollary}{theorem}
\newtheorem{corollary}[corollary]{Corollary}
\aliascntresetthe{corollary}

\newaliascnt{claim}{theorem}

\aliascntresetthe{claim}

\newaliascnt{problem}{theorem}
\newtheorem{problem}[problem]{Problem}
\aliascntresetthe{problem}

\newaliascnt{question}{theorem}
\newtheorem{question}[question]{Question}
\aliascntresetthe{question}

\newaliascnt{remark}{theorem}
\newtheorem{remark}[remark]{Remark}
\aliascntresetthe{remark}

\newaliascnt{example}{theorem}
\newtheorem{example}[example]{Example}
\aliascntresetthe{example}

\newaliascnt{notation}{theorem}
\newtheorem{notation}[notation]{Notation}
\aliascntresetthe{notation}

\begin{document}
\normalsize
\title{{\bf Adjoints of Matroids}}
\author{Houshan Fu,\;Chunming Tang\thanks{Corresponding author and supported by NSFC 12171114}\\
\small School of Mathematics and Information Science\\
\small Guangzhou University, Guangzhou, China\\
\small fuhoushan@gzhu.edu.cn,\;ctang@gzhu.edu.cn
\and
Suijie Wang\thanks{Supported by NSFC 12171487}\\
\small School of Mathematics\\
\small Hunan University, Changsha, China\\
\small  wangsuijie@hnu.edu.cn}
\date{}
\maketitle
\begin{abstract}
We show that an adjoint of a loopless matroid  is connected if and only if it itself  is connected. Our first goal is to study the adjoint of  modular matroids. We prove that a modular matroid has only one adjoint (up to isomorphism) which can be given by its opposite lattice, and proceed to present some alternative characterizations of  modular matroids associated to adjoints and opposite lattices. The other purpose is to investigate the adjoint sequence $ad^0M,adM,ad^2M,\ldots$ of a  connected matroid $M$. We classify such adjoint sequences into three types: finite, cyclic and convergent.  For the first two types, the adjoint sequences eventually stabilize at the finite projective geometries except for free matroids. For the last type, the infinite non-repeating adjoint sequences are convergent to the infinite projective geometries.
\vspace{1ex}\vspace{1ex}\\
MSC classes: 05B35.
\end{abstract}
\section{Introduction}\label{Sec-1}
To explore the incidence relations among flats of a geometric lattice, Cheung \cite{Cheung1974} introduced adjoint of the geometric lattice in 1974. Given a geometric lattice $\mathcal{L}$, its opposite lattice $\mathcal{L}^{op}$ obtained from $\mathcal{L}$ by reversing the order relation is not necessarily geometric lattice. If $\mathcal{L}^{op}$ can be embedded into a geometric lattice $\mathcal{L}^\Delta$ and this embedding sends all atoms of $\mathcal{L}^{op}$ bijectively onto atoms of $\mathcal{L}^\Delta$, then $\mathcal{L}^\Delta$ is called an adjoint of $\mathcal{L}$. Given a matroid $ M$, all flats of $M$ automatically form a geometric lattice $\mathcal{L}(M)$. A matroid $N$ is said to be an adjoint of $M$ if $\mathcal{L}(N)$ is an adjoint of $\mathcal{L}(M)$, which was extended to an oriented matroid by Bachem and Kern \cite{Bachem-Kern1986} in 1986. In general, a matroid with rank no more than three has always an adjoint. Nevertheless, when the rank of a matroid is greater than three, it may fail to admit an adjoint, such as, Cheung \cite{Cheung1974} showed that V\'{a}mos matroid does not have an adjoint. Even though a matroid has an adjoint, it is not necessarily unique. The existence and uniqueness of an adjoint is always an important topic in the field. For example,  Alfter and  Hochst\"attler\cite{Alfter-Hochstattler1995} presented a class of matroids of rank four with adjoint including a non-linear example; Alfter, Kern and Wanka \cite{Alfter-Kern-Wanka1990} found that a non-Desargues matroid admits an adjoint, and its dual matroid fail to have an adjoint. More results related to this topic, see \cite{Bachem-Wanka1989,Bjorner1999,Hochstattler-Kromberg1996,Hochstattler-Kromberg1996-1}.

In general, the extension lattice of a matroid is a very important tool to address this issue of  the existence and uniqueness of adjoints. See \cite{Bachem-Kern1986-1} for a further discussion of extension lattices. \cite{Alfter-Kern-Wanka1990} and \cite{Cheung1974} provided an alternative characterization of adjoints that a matroid $N$ is an adjoint of a matroid $M$ if and only if the geometric lattice $\mathcal{L}(N)$ of $N$ can be embedded into the extension lattice $\mathcal{E}(M)$ of $M$, taking the atoms of $\mathcal{L}(N)$ bijectively onto the atoms of $\mathcal{E}(M)$. From the lattice theory, the well-known point-hyperplane duality from projective geometry tells us that the extension lattice of a connected modular matroid may be regarded as its dual geometry.  On the other hand,  the opposite lattice of  a connected modular matroid determines a projective geometry (matroid), which is also its an adjoint. Then a natural question is whether the dual geometry obtained from the hyperplanes of a connected modular matroid is isomorphic to the projective geometry given by its opposite lattice. If this holds,  the next question arises: for a fixed modular matroid, whether the matroid determined by its opposite lattice is the unique adjoint or not up to isomorphism. This poses a further question: what kinds of matroids have only one adjoint up to isomorphism except for the modular matroids. In this paper, we make an attempt to answer the first two questions by proving that the opposite lattice of a connected modular matroid is isomorphic to its extension lattice. The last question seems likely to be not easy and could not have been solved.

For representable matroids, there are some additional explanations for adjoints. When $M$ is a vector matroid, Bixby and Coullard \cite{Bixby-Coullard1988} constructed  its an adjoint $\sigma M$ in two equivalent ways: one is from cocircuits of $M$, the other is from hyperplanes of $M$, called the type I adjoint. This implies that if a matroid is representable, then it  has always an adjoint. Subsequently, Hochst\"attler and Kromberg \cite{Hochstattler-Kromberg1996} made a detailed study of the type I adjoint.  In 2015, Jurrius and Pellikaan \cite{Jurrius-Pellikaan2015} obtained the type I adjoint from the derived code given by the cocircuits. Most recently,  Kung \cite{Kung2020} made a more profound and full scale introduction \cite{Kung2020} for cocircuit matroid that is exactly the type I adjoint.

Another related notion is derived matroid.  Initially, Rota proposed the program of finding the dependencies on the circuits of a matroid at the Bowdoin College Summer 1971 NSF Conference on Combinatorics. In 1980, Longyear \cite{Longyear1980} developed Rota's idea by introducing the derived matroid for binary matroid.  Recently, Oxley and Wang \cite{Oxley-Wang2019} further gave a more general definition of the derived matroid for arbitrary representable matroids by the circuits, which is referred to as Oxley-Wang derived matroid in \cite{Freij2023}. Noting that there is a one-to-one correspondence between hyperplanes of  a matroid $M$ and circuits of its dual matroid $M^*$. For a fixed $\mathbb{F}$-represented matroid $M$, this leads to the duality relation between the type I adjoint $\sigma M$ of $M$ and the Oxley-Wang derived matroid $\delta_{OW}M^*$ of $M^*$, that is,
\begin{equation}\label{Adjoint-Derived}
\sigma M\cong\delta_{OW}M^*.
\end{equation}
A geometric interpretation of the above relation can refer to \cite{Freij2023}. One very significant work in \cite{Oxley-Wang2019} is that Oxley and Wang classified all Oxley-Wang derived sequences into three types: finite, cyclic, and divergent. Motivated by this, we shall try our best to describe the general adjoint sequences that may be not easy. As a byproduct,  we shall further classify all type I adjoint sequences associated with an $\mathbb{F}$-represented matroid, which is implicit in the work of Kung \cite{Kung2020}. The earlier results may see literatures  \cite{Bjorner1999,Jurrius-Pellikaan2015}. Additionally, Freij-Hollanti, Jurrius and Kuznetsova \cite{Freij2023} constructed the combinatorial derived matroid for arbitrary matroids and provided some open questions and more ideas for further research.

Our paper is organized as follows. We start by introducing the preliminaries on adjoints in \autoref{Sec-2}.  \autoref{Sec-3} is devoted to showing that when a matroid has no loops,  its an adjoint is connected if and only if it itself is connected. We further verify that the existence of an adjoint of a matroid is completely determined by the existence of an adjoint of each connected submatroid. In \autoref{Sec-4}, we focus on investigating the adjoins of modular matroids. We show that the opposite lattice of a modular matroid can be identified with its extension lattice. This implies that a modular matroid has a unique adjoint up to isomorphism, and also yields some alternative characterizations of modular matroids associated to adjoints, opposite lattices and so on. \autoref{Sec-5} is concentrated on classifying adjoint sequences. Roughly speaking,  an adjoint sequence of a connected matroid will usually end with a projective geometry, please see \autoref{Finite-Type}, \autoref{Cyclic-Type} and \autoref{Convergent-Type} for more detailed information. As a byproduct of \autoref{Sec-5}, \autoref{Sec-6} characterizes the type I adjoint sequences as well as gives a detailed proof of the duality relation in \eqref{Adjoint-Derived}. In \autoref{Sec-7}, we provide another description of the adjoint by the cocircuits which leads to more ideas for further research associated with combinatorial derived matroids in \cite{Freij2023}.
\section{Preliminaries on adjoints}\label{Sec-2}
The matroid terminology and notations will follow Oxley's book \cite{Oxley}.  Throughout this paper, we always assume that matroids  are non-empty and finite, unless otherwise mentioned. We shall present some necessary notations  and omit their detailed explanations.
\begin{notation}
{\rm
Let $M$ be a matroid. Unless explicitly mentioned otherwise, $E(M)$, $\mathcal{I}(M)$, $\mathcal{H}(M)$, $\mathcal{C}(M)$ and $\mathcal{L}(M)$ denote the ground set of $M$, the set of all independent sets of $M$, the set of all hyperplanes of $M$, the set of all circuits of $M$ and the set of  all flats of $M$ in turn. $r_{M}(\cdot)$ is the rank function of $M$. Notations $\wedge_\mathcal{L}$, $\vee_\mathcal{L}$ and $\le_\mathcal{L}$ denote the `meet', `joint'  and `partial order' in the lattice $\mathcal{L}$,  respectively. Notation $\sqcup$ denotes the disjoint union of sets.
}
\end{notation}

The original definition of  the adjoint of a matroid  is defined as follows.
\begin{definition}\label{Adjoint-Def}
{\rm
Let $ M$ be a matroid. A matroid $ adM$ is an {\em adjoint} of $ M$ if $ r( adM)= r( M)$, and there is an injection and order-reversing map $\phi:\mathcal{L}(M)\to \mathcal{L}(adM)$  sending coatoms of $\mathcal{L}(M)$ bijectively onto atoms of $\mathcal{L}(adM)$.
}
\end{definition}
 If an adjoint $adM$ of a matroid $M$ exists, we call the map  $\phi:\mathcal{L}(M)\to \mathcal{L}(adM)$ in \autoref{Adjoint-Def} an {\em adjoint map} of $M$. The adjoint map $\phi$ shows how the embedding works. Given any flat $X$ of $M$, the adjoint map $\phi$ sends  $X$ to
 \begin{equation}\label{Adjoint-Map}
 \phi(X)=\bigvee_{H\in \mathcal{H}(M),\,X\subseteq H}\phi(H)=\bigcup_{H\in \mathcal{H}(M),\,X\subseteq H}\phi(H).
 \end{equation}
Associated with the adjoint map $\phi$, below provides more fundamental properties of an adjoint, which may be used later.
\begin{proposition}[\cite{Bixby-Coullard1988,Hochstattler-Kromberg1996-1}]\label{Adjoint-Prop}
Let $ M$ be a matroid and $adM$ be its an adjoint. For any flats $X,\;Y\in\mathcal{L}(M)$,  the adjoint map $\phi:\mathcal{L}(M)\to \mathcal{L}(adM)$ has the following properties:
\begin{itemize}
  \item [\rm(i)] If $X$ covers $Y$ in $\mathcal{L}(M)$, then $\phi(Y)$ covers $\phi(X)$ in $\mathcal{L}(adM)$.
  \item [\rm(ii)] $ r_{adM}\big(\phi(X)\big)= r( M)- r_{M}(X)$.
  \item [\rm(iii)]$\phi(X)\wedge_{\mathcal{L}(adM)}\phi(Y)=\phi(X\vee_{\mathcal{L}(M)} Y)$.
  \item [\rm(iv)]  $r_{adM}\big(\phi(X)\big)+ r_{adM}\big(\phi(Y)\big)= r_{adM}\big(\phi(X)\wedge_{\mathcal{L}(adM)}\phi(Y)\big)+ r_{adM}\big(\phi(X)\vee_{\mathcal{L}(adM)}\phi(Y)\big)$.
   \item [\rm(v)] If $H_1,H_2,\ldots, H_m\in\mathcal{H}(M)$ satisfy
   \[
  H_1\cap\cdots\cap H_m\subsetneqq H_1\cap\cdots\cap H_{m-1}\subsetneqq\cdots\subsetneqq H_1\cap H_2\subsetneqq H_1,
   \]
   then $\phi(H_1),\phi(H_2),\ldots,\phi(H_m)$ is an independent set of $adM$.
\end{itemize}
\end{proposition}

It is clear that each adjoint is simple. Moreover,  the loops and parallel elements of  a matroid have no effect on its adjoints. Thus, we have the following fact.
\begin{fact}\label{Fact0}
{\rm Let $adM$ be an adjoint of a matroid $M$. If $e$ is a loop or a parallel element,  then $adM$ is also an adjoint of $M\setminus e$.}
\end{fact}

In end of this section,  let us quickly recall the type I adjoint for representable matroids. For a field $\mathbb{F}$, given an $\mathbb{F}$-represented matroid $(M,A):=M[A]$ on ground set $E(M)=\{e_1,e_2,\ldots,e_m\}$ of rank $r$, where the columns of  the matrix $A\in \Bbb{F}^{r\times m}$ are labelled, in order, $e_1,e_2,\ldots,e_m$. Each hyperplane $H$ of $M$ naturally determines a hyperplane ${\rm span}\{A_{e_i}:e_i\in H\}$ in $\mathbb{F}^r$ that is spanned by all columns in $A$ labelled by $e_i\in H$.  Let ${\bm h}_H$ be the normal vector of ${\rm span}\{A_{e_i}:e_i\in H\}$ in $\Bbb{F}^r$ when $H\in \mathcal{H}(M)$. The {\em type I adjoint} $\sigma M$ of $M$ is defined as
\begin{equation}\label{Adjoint-I}
\sigma M:=M\b[{\bm h_H}\mid H\in \mathcal{H}(M)\b].
\end{equation}
\section{Connectivity}\label{Sec-3}
We are now ready to study the connectivity of adjoints. To this end, we first introduce another way to characterize an adjoint of a matroid $M$ associated with those sets, such a set consists of its all hyperplanes containing a fixed element of $M$ except for loops.
\begin{proposition}\label{Adjoint-Cha}
Let $ M$ and $adM$ be two matroids of the same rank $r$. Then $adM$ is an adjoint of $M$ if and only if its ground set can be regarded as $E(adM):=\mathcal{H}(M)$ such that the sets $H[e]:=\{H\in\mathcal{H}(M)\mid e\in H\}$ are hyperplanes of $adM$ for all $e\in E(M)$ except for loops.
\end{proposition}
\begin{proof}
Based on \autoref{Fact0}, we may assume that $M$ is simple. For the sufficiency,  if $adM$ is an adjoint of $M$,  then the equation \eqref{Adjoint-Map} implies that the adjoint map $\phi:\mathcal{L}(M)\to\mathcal{L}(adM)$ sends each $e\in E(M)$  to $\phi(e)=\bigcup_{H\in \mathcal{H}(M),\,e\in H}\phi(H)\in\mathcal{L}(adM)$. Immediately, we have $r_{adM}(\phi(e))=r-1$  from the part ${\rm(ii)}$ in \autoref{Adjoint-Prop}. Namely, every $\phi(e)$ is a hyperplane of $adM$. In this case, suppose $E(adM):=\mathcal{H}(M)$ and $\phi(H)=H$ for $H\in\mathcal{H}(M)$. Then we obtain that $H[e]=\phi(e)$ is a hyperplane of $adM$.

For the necessity,  we define a map $\psi:\mathcal{L}(M)\to\mathcal{L}(adM)$ such that $\psi$ sends $e\in E(M)$ to $\psi(e)=H[e]$ and each flat $X$ of $M$ to $\psi(X)=\bigcap_{e\in X}H[e]$. Obviously, $\bigcap_{e\in X}H[e]\in \mathcal{L}(adM)$ since $H[e]$ is a hyperplane of $adM$. So $\psi$ is well defined.  Note from $E(adM):=\mathcal{H}(M)$ that $\psi$  automatically becomes  a bijection between the coatoms of $\mathcal{L}(M)$ and the atoms of $\mathcal{L}(adM)$.  To prove that $adM$ is an adjoint of $M$, it suffices to verify that $\psi$ is an order-reversing and injective map. Given any flats $X,Y\in\mathcal{L}(M)$. If  $X\subseteq Y$ in $\mathcal{L}(M)$, then we arrive at
\[
\psi(Y)=\bigcap_{e\in Y}H[e]\supseteq\bigcap_{e\in X}H[e]=\psi(X).
\]
Namely, $\psi$ is an order-reversing map. Recall from the definition of $H[e]$ that
\[
\psi(X)=\bigcap_{e\in X}H[e]=\bigcap_{e\in X}\{H\in\mathcal{H}(M)\mid e\in H\}=\{H\in\mathcal{H}(M)\mid X\subseteq H\}.
\]
Likewise, we get $\psi(Y)=\{H\in\mathcal{H}(M)\mid Y\subseteq H\}$. This implies that if $X\ne Y$, then $\psi(X)\ne\psi(Y)$. Hence, $\psi$ is injective. We complete this proof.
\end{proof}

Given a basis $B$ of  a matroid $M$ and an element $e$ of $B$, define $H(e;B)$ is a hyperplane of $M$ satisfying $B\backslash e\subseteq H(e;B)$, the unique hyperplane $H(e;B)$ is called the {\em fundamental hyperplane} of $e$ with respect to $B$. Next we will use the equivalent characterization of adjoints in \autoref{Adjoint-Cha} to show that  all fundamental hyperplanes of a matroid with respect to its a fixed basis form a basis of its an adjoint.
\begin{lemma}\label{Ajoint-basis}
Let $M$ be a matroid of rank $r$ and $adM$ its an adjoint with ground set $E(adM)=\mathcal{H}(M)$. If $B$ is a basis of $ M$, then $\{H(e;B)\mid e\in B\}$ is a basis of $adM$.
\end{lemma}
\begin{proof}
Without loss of generality, let the basis $B=\big\{e_i\mid i\in[r]\big\}$ of $M$.  Noting from the definition of the fundamental hyperplane that for any $k\in[r]$,  all fundamental hyperplanes $H(e_i;B)$ contain $e_k$ except for $H(e_k;B)$. This indicates
\[
  H(e_1;B)\cap\cdots\cap  H(e_r;B)\subsetneqq  H(e_1;B)\cap\cdots\cap H(e_{r-1};B)\subsetneqq\cdots\subsetneqq  H(e_1;B)\cap  H(e_2;B)\subsetneqq  H(e_1;B).
 \]
It follows from the property ${\rm (v)}$ in \autoref{Adjoint-Prop} that $\big\{H(e_i;B)\mid i\in[r]\big\}$ is independent. Then $\big\{H(e_i;B)\mid i\in[r]\big\}$ is a basis of $adM$ since the rank of $adM$ equals $r$.
\end{proof}

The following result states that if an adjoint $adM$ is not connected, then the original matroid $M$ is also not connected.
\begin{theorem}\label{Adjoint-Con1}
Let $M$ be a loopless matroid and $adM$ its an adjoint.  If $adM=N_1\oplus N_2$, then there are submatroids $M_1$ and $M_2$ of $M$ such that $M=M_1\oplus M_2$ and $N_i=ad M_i$ for $i=1,2$.
\end{theorem}
\begin{proof}
We need only consider that $M$ is simple from \autoref{Fact0}.  By \autoref{Adjoint-Cha}, we may assume that $E(adM)=E(N_1)\sqcup E(N_2)=\mathcal{H}(M)$,  and the adjoint map $\phi:\mathcal{L}(M)\to\mathcal{L}(adM)$ sends each element $e\in E(M)$ to $\phi(e)=H[e]$. Then $adM=N_1\oplus N_2$ indicates that for each cocircuit $C^*$ of $adM$, $C^*$ is contained in one of $E(N_1)$ and $E(N_2)$. It means that $E(N_1)$ and $E(N_2)$ automatically induce a partition of $\mathcal{H}(adM)$ such that $H\in\mathcal{H}_{E(N_i)}$ if and only if $H\in\mathcal{H}(adM)$ and $E(N_i)\subseteq H$ for $i=1,2$.  From the injectivity of $\phi$, let $E_i=\big\{e\in E(M)\mid H[e]\in\mathcal{H}_{E(N_i)}\big\}$ and $M_i=M\setminus E_i$ for each $i$. Then $E(M)=E_1\sqcup E_2$, $E(M_1)=E_2$ and $E(M_2)=E_1$.  We claim that $E_1$ and $E_2$ are not empty set. Suppose $E_1=\emptyset$.  Namely, $E_2=E(M)$. Then $H[e]\in \mathcal{H}_{E(N_2)}$ for all $e\in E(M)$. Recall from the definition of $\mathcal{H}_{E(N_2)}$ that $E(N_2)\subseteq H[e]$ for each $e\in E(M)$. This implies that $E(M)\subseteq H$ for each hyperplane $H$ of $M$ in $E(N_2)$,  a contradiction. Then we have verified $E_1\ne\emptyset$. Likewise, we also obtain $E_2\ne\emptyset$.

To prove $M=M_1\oplus M_2$, it is equivalent to showing $r_M(E_1)+r_M(E_2)=r(M)$. Let $r(M)=r$ and  $B=\{e_1,e_2\ldots,e_r\}$ be a basis of $M$. Using the same argument as in the proof of that $E_1\ne\emptyset$, we can arrive at $B\cap E_i\ne\emptyset$ for $i=1,2$ as well. Suppose $B_1=B\cap E_1=\{e_1,\ldots,e_k\}$ and $B_2=B\cap E_2=\{e_{k+1},\ldots,e_r\}$ for some positive integer $ k<r$. Recall the definitions of $E_1$ and $E_2$, we have
$B_1\subseteq H_1$ and $B_2\subseteq H_2$ for all hyperplanes $H_1\in E(N_1)$ and $H_2\in E(N_2)$ of $M$. It follows from the definition of the fundamental hyperplane that we obtain $H(e_i;B)\in E(N_2)$ for all  $e_i\in B_1$ and $H(e_i;B)\in E(N_1)$ for all $e_i\in B_2$.  According to \autoref{Ajoint-basis},
$r_{adM}(E(N_1))=|B_2|=r-k$ and $r_{adM}(E(N_2))=|B_1|=k$ since $adM=N_1\oplus N_2$. We assert $r_M(E_1)=k$ and $r_M(E_2)=r-k$. Otherwise, we may assume $r_M(E_1)=j\ne k$. Let $B_1'$ be a basis of $E_1$ and $B'$ a basis of $E(M)$ containing $B_1'$. Repeating the same argument as above, we get $r_{adM}(E(N_2))=|B_1'|=j\ne k$,  a contradiction. Thus, we obtain  $r_M(E_1)+r_M(E_2)=r$.

Moreover, notice from $M=M_1\oplus M_2$ that the intervals $[\emptyset,E_2]$ and $[E_1,E(M)]$ of $\mathcal{L}(M)$ are isomorphic under the map $\phi_1$, where $\phi_1:[\emptyset,E_2]\to [E_1,E(M)]$ sends $X$ to $X\sqcup E_1$. Then $\phi\circ\phi_1:\mathcal{L}(M_1)\to\mathcal{L}(adM)$ is injective, order-reversing, and onto $E(N_1)$. Thus $N_1=adM_1$. Likewise, we can verify $N_2=adM_2$ as well. This completes the proof.
\end{proof}

On the other hand, Bixby and Coullard \cite{Bixby-Coullard1988} proved the opposite of \autoref{Adjoint-Con1} that an adjoint of a direct sum of two matroids is the direct sum of the adjoints of these matroids.
\begin{lemma}[\cite{Bixby-Coullard1988}, Lemma 4.2]\label{Adjoint-Con2}
Let $M$ be a loopless  matroid and $adM$ its an adjoint. If  $M=M_1\oplus M_2$, then there are submatroids  $N_1,N_2$ of $adM$ such that $adM=N_1\oplus N_2$, where $N_i=adM_i$ for $i=1,2$.
\end{lemma}

The following result is the straightforward consequence of  \autoref{Adjoint-Con1} and \autoref{Adjoint-Con2}.
\begin{corollary}\label{Adjoint-Con3}
Let $M$ be a loopless matroid and $ adM$ its an adjoint. Then $M$ is connected if and only if $adM$ is connected.
\end{corollary}

Below explains a close connection between the existence of  an adjoint and the existence of an adjoint of  every connected submatroid of the original matroid.
\begin{corollary}\label{Adjoint-Minor}
Let $M$ be a loopless matroid and write as a direct sum of  its connected components $M_1,\ldots, M_n$. Then $M$ has an adjoint if and only if each connected component $M_i$ has an  adjoint.
\end{corollary}
\begin{proof}
\autoref{Adjoint-Con1} has verified the sufficiency. For necessity, let $adM_i$ be an adjoint of $M_i$,  $\phi_i:\mathcal{L}(M_i)\to\mathcal{L}(adM_i)$ be the adjoint map for each $i\in[n]$, and $adM=adM_1\oplus\cdots\oplus adM_n$. Define a map $\phi:\mathcal{L}(M)\to\mathcal{L}(adM)$
such that for any flat $X=\bigsqcup_{i=1}^nX_i$ with $X_i\in\mathcal{L}(M_i)$, $\phi(X)=\bigsqcup_{i=1}^n\phi_i(X_i)$. It is clear that $r(M)=\sum_{i=1}^nr(M_i)=\sum_{i=1}^nr(adM_i)=r(adM)$. Moreover, all adjoint maps $\phi_i$ guarantee that $\phi$ is injective, order-reversing as well as onto $E(adM)$. So $adM$ is an adjoint of $M$.
\end{proof}
\section{Modular matroids}\label{Sec-4}
In this section, we shall focus on investigating modular matroids and their adjoints. Intuitively, the construction of an adjoint is closely related to modular matroids. The procedure from a matroid to its an adjoint will produce many more modular pairs and even many more modular flats. More precisely, taking two distinct flats $X$ and $Y$ of a matroid $M$, we know that their ranks satisfy the following submodular inequality
 \[
 r_M(X\vee_{\mathcal{L}(M)} Y)+r_M(X\wedge_{\mathcal{L}(M)} Y)\le r_M(X)+r_M(Y).
 \]
The equation does not always hold. When this holds, we call $(X,Y)$ a modular pair of flats. If all pairs of flats of $M$ are modular, $M$ is said to be a {\em modular matroid}. Thus,  an adjoint $adM$ of a matroid $M$ is constructed by adding some elements into the geometric lattice $\mathcal{L}(M)$ such that $r_{adM}(\phi(X)\vee_{\mathcal{L}(adM)}\phi(Y))+r_M(\phi(X)\wedge_{\mathcal{L}(adM)}\phi(Y))= r_{adM}(\phi(X))+r_{adM}(\phi(Y))$ for all flats $X,Y\in\mathcal{L}(M)$.

First noting the fact that if $M$ is a modular matroid, then the opposite lattice $\mathcal{L}(M)^{op}$ is a geometric lattice. Immediately, $\mathcal{L}(M)^{op}$ determines a matroid $adM$ that is exactly an adjoint of $M$ in this case.  Conversely, if a matroid $M$ has an adjoint $adM$ such that $\mathcal{L}(adM)\cong\mathcal{L}(M)^{op}$. Then $\mathcal{L}(M)^{op}$ is a geometric lattice. This further implies that $M$ is modular in the case. An immediate result from the preceding arguments is stated as folows.
\begin{proposition}\label{Modular-Unique}
Let $M$ be a simple matroid. Then $M$ is modular  if and only if  $M$ has an adjoint $adM$ such that $\mathcal{L}(adM)\cong\mathcal{L}(M)^{op}$. 
\end{proposition}

\autoref{Modular-Unique} states that a modular matroid has always an adjoint that is isomorphic to its opposite lattice. It is natural to ask if for some fixed modular matroid $M$, the adjoint given by the opposite lattice $\mathcal{L}(M)^{op}$ is a unique adjoint of $M$ up to isomorphism. To answer this problem, we need to introduce linear subclass and extension lattice of a matroid. In 1965, Crapo \cite{Crapo1965} used the linear subclasses of a matroid to characterize its all single-element extensions. A {\em linear subclass} $\mathcal{H}$ of  a matroid $M$ is a subset of its hyperplanes with the following property: if $H_1$ and $H_2$ are the members of $\mathcal{H}$ such that $r_M(H_1\cap H_2)=r(M)-2$, and $H_3$ is a hyperplane containing $H_1\cap H_2$, then $H_3\in\mathcal{H}$. All linear subclasses of $M$ form a lattice ordered by inclusion, called an {\em extension lattice} of $M$ and denoted by $\mathcal{E}(M)$.

In general, posets $\mathcal{P}_1$ and $\mathcal{P}_2$ are isomorphic each other if and only if there is a bijection $\theta:\mathcal{P}_1\to\mathcal{P}_2$ such that for any members $X,Y$ of $\mathcal{P}_1$,
\[
X\le_{\mathcal{P}_1} Y \mbox{ if and only if } \theta(X)\le_{\mathcal{P}_2}\theta(Y),
\]
denoted by $\mathcal{P}_1\cong\mathcal{P}_2$.
When $M$ is modular, the following result will show that the opposite lattice  $\mathcal{L}(M)^{op}$ of $M$ is isomorphic to its extension lattice $\mathcal{E}(M)$. Now let us recall the intersection property of a modular geometric lattice, which will be a crucial tool to prove the surjectivity in the proof of \autoref{Adjoint-Extension}. A modular geometric lattice $\mathcal{L}$ meets the {\em intersection property}: for any two distinct atoms $X,Y$ of $\mathcal{L}$ and an element $Z\in\mathcal{L}$, if $X\le_\mathcal{L}Y\vee_\mathcal{L}Z$, then there exists an atom $W$ of $\mathcal{L}$ such that $W\le_\mathcal{L}(X\vee_\mathcal{L}Y)\wedge_\mathcal{L}Z$. We are now turning to the following key lemma.
\begin{lemma}\label{Adjoint-Extension}
Let $M$ be a simple matroid of rank $r$.  If $M$ is modular, then the opposite lattice $\mathcal{L}(M)^{op}$ of $M$ is isomorphic to its extension lattice $\mathcal{E}(M)$.
\end{lemma}
\begin{proof}
Define a map $\lambda: \mathcal{L}(M)^{op}\to\mathcal{E}(M)$ sending each member  $X$ in $\mathcal{L}(M)^{op}$ to $\lambda(X)=\mathcal{H}_X(M)$, where $\mathcal{H}_X(M)=\{H\in\mathcal{H}(M)\mid X\subseteq H\}$. We begin by proving that $\lambda$ is well defined. For the both cases $r_M(X)=r,\,r-1$, $\mathcal{H}_X(M)$ are the linear subclasses of $M$ obviously. Note that $\mathcal{L}(M)^{op}$ is a modular geometric lattice since $M$ is modular. For $r_M(X)\le r-2$, we have $r_M(H_1\cap H_2)=r_M(H_1)+r_M(H_2)-r_M(H_1\vee_{\mathcal{L}(M)} H_2)=r-2$ for any distinct members $H_1,H_2\in\mathcal{H}_X(M)$.  With further step,  if $H_1\cap H_2\subseteq H_3$, we have $X\subseteq H_1\cap H_2\subseteq H_3$. Then $H_3\in\mathcal{H}_X(M)$ and $\mathcal{H}_X(M)$ is a linear subclass in this case. So $\lambda$ is well defined.

To prove $\mathcal{L}(M)^{op}\cong\mathcal{E}(M)$, we shall show that $\lambda$ is order-preserving, injective and surjective in turn. If $X\le_{\mathcal{L}(M)^{op}} Y$ in $\mathcal{L}(M)^{op}$, then we have $Y\subseteq X$. This means that if the hyperplane $H$ of $M$ contains $X$, then $H$ must contain $Y$. Namely, $\mathcal{H}_X(M)\subseteq \mathcal{H}_Y(M)$ in $\mathcal{E}(M)$. So $\lambda$ is an order-preserving map. Suppose $X\ne Y$, then $\mathcal{H}_X(M)\ne\mathcal{H}_Y(M)$ obviously. That is, $\lambda$ is an injection.

Now we are ready to prove the surjectivity of $\lambda$. Given a linear subclass $\mathcal{H}\subseteq\mathcal{H}(M)$ of $M$, let $X_\mathcal{H}=\bigcap_{H\in\mathcal{H}}H$. Obviously, $X_\mathcal{H}\in \mathcal{L}(M)^{op}$. It remains to verify $\lambda(X_{\mathcal{H}})=\mathcal{H}_{X_\mathcal{H}}=\mathcal{H}$.
The both cases $r_M(X_\mathcal{H})=r,r-1$ are trivial. So let $r_M(X_\mathcal{H})=k$ for some $0\le k\le r-2$.
According to the definition of  $\mathcal{H}_{X_\mathcal{H}}$, we arrive at $\mathcal{H}\subseteq\mathcal{H}_{X_\mathcal{H}}$.
Next we will show $\mathcal{H}\supseteq\mathcal{H}_{X_\mathcal{H}}$  in the case.  Suppose existing $H_0\in\mathcal{H}_{X_\mathcal{H}}\setminus\mathcal{H}$ satisfies that $H\cap H'$ is not contained in $H_0$ for any $H,H'\in\mathcal{H}$. Then from $X_{\mathcal{H}}\subseteq H_0$ we can choose a hyperplane sequence $H_1,H_2,\ldots,H_{r-k}$ of $M$ such that
\[
X_\mathcal{H}= \bigcap_{j=1}^{r-k}H_j\lessdot\bigcap_{j=1}^{r-k-1}H_j\lessdot\cdots\lessdot \bigcap_{j=1}^{i_0-1}H_j\lessdot \bigcap_{j=1}^{i_0}H_j\lessdot\cdots\lessdot H_1\bigcap H_2\lessdot H_1
\]
and $i_0$ is the minimal positive number for which
\begin{equation}\label{Intersction}
\bigcap_{j=1}^{i_0-1}H_j\nsubseteq H_0\quad\quad\And\quad\quad\bigcap_{j=1}^{i_0}H_j\subseteq H_0.
\end{equation}
Obviously, we have $i_0\ge 3$ and $\bigcap_{j=1}^{i_0}H_j\subseteq\bigcap_{j=1}^{i_0-1}H_j\bigcap H_0$. From the modularity of $M$, we know that $\mathcal{L}(M)^{op}$ is a modular geometric lattice and  $\mathcal{H}(M)$ is the set of all its atoms. So $\mathcal{L}(M)^{op}$ has the intersection property. From this property, we obtain from \eqref{Intersction} that there is a hyperplane $H_{0i_0}$ of $M$ such that
\[
H_{i_0}\bigcap H_0\subseteq H_{0i_0}\quad\quad\And\quad\quad\bigcap_{j=1}^{i_0-1}H_j\subseteq H_{0i_0}.
\]
If $H_{0i_0}\in\mathcal{H}$, then we can get $H_{i_0}\cap H_{0i_0}\subseteq H_0$. That is, $H_0\in\mathcal{H}$, a contradiction. Hence, $H_{0i_0}\in\mathcal{H}_{X_\mathcal{H}}\setminus\mathcal{H}$ since $X_{\mathcal{H}}\subseteq H_{i_0}\cap H_0\subseteq H_{0i_0}$. Moreover, the minimality of $i_0$ guarantees $H_{0i_0}\ne H_0$. Then we can replace $H_0$ with $H_{0i_0}$. Using the same argument as in the case $H_0$, we can arrive at a positive integer $i_1<i_0$ and $H_{i_0i_1}\ne H_0, H_{0i_0}$ having the same properties as $i_0$ and $H_{0i_0}$, respectively. The same step can go on and on. Then we can obtain infinitely many distinct integers $i_0,i_1,\ldots$ lying in the  interval $[3,i_0]$, which contradicts  the fact that this interval $[3,i_0]$ contains only finitely many integers. The proof is completed.
\end{proof}

\autoref{Adjoint-Extension} indicates that a modular matroid has only one adjoint up to isomorphism, which is given by its opposite lattice. The property makes modular matroids become a key ingredient in the classification problem of adjoint sequences later.
\begin{theorem}\label{Modular-Unique-1}
Let $M$ be a simple matroid. If $M$ is modular, then $M$ has only one adjoint $adM$ up to isomorphism and $\mathcal{L}(adM)\cong\mathcal{L}(M)^{op}$. 
\end{theorem}
\begin{proof}
Recall from \autoref{Modular-Unique} that the modular matroid $M$ has always an adjoint. Given an adjoint $adM$ of $M$. Recall from \cite{Cheung1974}  that $\mathcal{L}(adM)$ can be viewed as a sublattice of the extension lattice $\mathcal{E}(M)$. Combining with the fact that the opposite lattice $\mathcal{L}(M)^{op}$ of $M$ is always regarded as a sublattice of $\mathcal{L}(adM)$. Then, when $M$ is modular,  $\mathcal{L}(M)^{op}\cong\mathcal{E}(M)$ in \autoref{Adjoint-Extension} directly leads to that $\mathcal{L}(adM)\cong\mathcal{L}(M)^{op}$. This further implies that $M$ has only one adjoint up to isomorphism.
\end{proof}

From the classical projective geometry \cite{Birkhoff1967}, we know that connected modular matroids can be identified with projective geometries except for free matroids.  Let $P$ and $L$ be the disjoint set of points and lines, respectively, and $\iota$ is an incidence relation between the points and the lines. If the triple $(P,L,\iota)$ satisfies the following incidence relations
\begin{itemize}
  \item [\rm(p1)] Every two distinct points, $a$ and $b$, are on exactly one line $a\iota b$,
  \item [\rm(p2)] Every line contains at least three points,
  \item [\rm(p3)] If $a,b,c$ and $d$ are four distinct points, no three of which are collinear, and if the line $a\iota b$ intersects the line $c\iota d$, then the line $a\iota c$ intersects the line $b\iota d$,
\end{itemize}
then the triple $(P,L,\iota)$ is called a {\em projective geometry}. In particular,  the simple matroid associated with the vector space $\mathbb{F}^r$ is a projective geometry, denoted by $PG(r-1,\mathbb{F})$. When $\mathbb{F}$ is the finite field $GF(q)$ with $q$ elements, it is customary to denote this projective geometry by $PG(r-1,q)$.

A subspace of a projective geometry $(P,L,\iota)$  is a subset $P'\subseteq P$ such that if $a$ and $b$ are distinct points of $P'$, then all points on the line $a\iota b$ are in $P'$. Let $\mathcal{L}(P)$ be the poset of all subspaces of  the projective geometry $(P,L,\iota)$, ordered by inclusion. From the classical lattice theory, we know that $\mathcal{L}(P)$ is a modular geometric lattice.  Let  $(P,L,\iota)$ be a projective geometry of rank $r$. The {\em dual geometry} $(P^*,L^*,\iota^*)$ of  $(P,L,\iota)$ is given by that all the coatoms of $\mathcal{L}(P)$ are regarded as the points of  $(P^*,L^*,\iota^*)$, and $a\iota^* b$ is a line in $L^*$ for any distinct points $a,b\in P^*$ if and only if $r_{\mathcal{L}(P)}(a\cap b)=r-2$.  In 1967, Birkhoff \cite{Birkhoff1967} showed that using the finite projective geometries can characterize the modular matroids.
\begin{proposition}[\cite{Birkhoff1967,Oxley}]\label{Modular-1}
Let $M$ be a simple matroid. Then $M$ is modular if and only if its every connected component is either the free matroid  $U_{1,1}$  or a finite projective geometry.
\end{proposition}

From this perspective, \autoref{Adjoint-Extension} implies that if $M$ is a connected modular matroid, then the extension lattice $\mathcal{E}(M)$ coincides with the lattice consisting of  all the subspace of the dual geometry of $M$. More specifically, hyperplanes of $M$ can be identified with points and members of $\mathcal{E}(M)$ with rank two can be regarded as lines in this case.  Immediately, these points, lines and inclusion relations between them will determine a projective geometry, which is precisely the dual geometry of $M$. Then the following proposition is a direct consequence of \autoref{Modular-Unique}, which may refer to \cite[11.2.3 Proposition]{Faure2000}.
\begin{proposition}
Let $(P,L,\iota)$ be a finite projective geometry. Then the lattice of its dual geometry is isomorphic to its opposite lattice. Moreover, the dual geometry of $(P,L,\iota)$ is a finite projective geometry.
\end{proposition}

Next we will describe the link between modular matroids and their adjoints. For this purpose, we require the classical Coordinatization Theorem  that a projective geometry of dimension at least three can be constructed as the projective geometry associated to a vector space over a field ,  which is also called Veblen-Young Theorem \cite{Veblen1965}.
\begin{theorem}[\cite{Veblen1965}, Coordinatization Theorem]\label{Veblen-Young}
Every projective geometry of rank $r\ge 4$ is isomorphic to $PG(r-1,\mathbb{F})$ for some field $\mathbb{F}$. In particular, every finite projective geometry of rank $r\ge4$ is isomorphic to $PG(r-1,q)$ for some finite field $GF(q)$ with $q$ elements.
\end{theorem}

\begin{remark}\label{Adjoint-Remark}
{\rm Let $M$ be a simple connected matroid of rank $r$. \autoref{Modular-1} and \autoref{Veblen-Young} indicate that if $M$ is a modular matroid of rank $r\ge4$, then  $M$ is isomorphic to the projective geometry $PG(r-1,q)$, and the unique adjoint $adM$ is always isomorphic to the type I adjoint $\sigma PG(r-1,q)$ of $PG(r-1,q)$.}
\end{remark}
\begin{theorem}\label{Modular-Adjoint}
Let $M$ be a simple matroid with no rank $3$ connected components. Then $M$ is  modular if and only if $M$ has an adjoint $adM$ such that $adM\cong M$.
\end{theorem}
\begin{proof}
For the necessity,  first noting from $adM\cong M$ that we have $\mathcal{L}(adM)\cong\mathcal{L}(M)^{op}$. This means that $\mathcal{L}(M)^{op}$ is a geometric lattice. So, $M$ is modular obviously.

For the sufficiency, from  \autoref{Adjoint-Con2} and \autoref{Adjoint-Minor}, we may assume that $M$ is connected with rank $r$. For $r=1$ and $2$, we have $M\cong U_{1,1}$ and $M\cong U_{2,m}$ for some positive integer $m\ge3$ from \autoref{Modular-1}. Then the both cases are trivial. For $r>3$, according to \autoref{Adjoint-Remark}, we need only consider this case that $M$ is the finite projective geometry $PG(r-1,q)$ for some finite field $GF(q)$ with $q$ elements, and $adM$ is identified with the type I adjoint $\sigma PG(r-1,q)$ of $PG(r-1,q)$. From the elementary linear algebra, there is a natural one-to-one correspondence between all $(r-1)$-dimensional subspaces and all $1$-dimensional subspaces of the vector space $GF(q)^r$ such that each $(r-1)$-dimensional subspace corresponds to its $1$-dimensional normal complement space in $GF(q)^r$. This automatically yields the next bijection
\begin{equation*}
\Psi: \mathcal{H}\big(PG(r-1,q)\big)\rightarrow E\big(PG(r-1,q)\big), \quad H\mapsto\Psi(H)=\bm h,
\end{equation*}
such that ${\rm span}(H)\oplus{\rm span}(\bm h)=GF(q)^r$ for $H\in\mathcal{H}\big(PG(r-1,q)\big)$. Combining with the definition of the type I adjoint in \eqref{Adjoint-I}, we obtain $PG(r-1,q)\cong\sigma PG(r-1,q)$  immediately.  We complete the proof.
\end{proof}

Let $M$ be a simple matroid. Suppose $adM$ is an adjoint of $M$. We conclude the preceding arguments by pointing out the following relations:
\[
adM\cong M\Leftrightarrow\mathcal{L}(M)\cong\mathcal{L}(M)^{op}\Rightarrow M\mbox{ is modular}\Leftrightarrow\mathcal{L}(adM)\cong\mathcal{L}(M)^{op}.
\]
When $M$ is modular, $\mathcal{L}(M)$ may be not isomorphic to $\mathcal{L}(M)^{op}$. For all we know, the main reason comes from the fact that some projective planes are not self-dual.

In end of this section, we further state that there is a close connection between the modularity of a matroid $M$ and the size of its an adjoint $adM$. Let $M$ be a simple matroid. Recall from \cite[Theorem 2]{Greene1970} that $M$ is modular if and only if $|\mathcal{H}(M)|=|E(M)|$.  Combining with \autoref{Modular-Unique}, we can obtain directly the following result.
\begin{proposition}\label{Modular-Card}
Let $M$ be a simple matroid.  Then $M$ is modular if and only if $M$ has  an adjoint $adM$ such that $|E(M)|=|E(adM)|$.
\end{proposition}
\section{Adjoint sequences}\label{Sec-5}

Inspired by Oxley and Wang's work in  \cite{Oxley-Wang2019}, we will study the classifications of adjoint sequences for arbitrary matroids.  Modular matroids will become the key ingredient to characterize adjoint sequences. Recall the arguments at the beginning of \autoref{Sec-4}, intuitively, an adjoint of a matroid $M$ is more close to modular matroid than $M$ itself. Thus a natural question arises: suppose a connected matroid $M$ has an adjoint sequence $ad^0M=M, adM,ad^2M,\ldots$, whether such adjoint sequence is eventually convergent to a projective geometry.  It is worth noting that \cite[Exercise 7.17]{Bjorner1999} seems to foreshadow this phenomenon.

Let $M$ be a matroid and $ad^0 M=M$. Inductively, for any positive integer $k$, the {\em $k$th adjoint} $ad^kM$ of $ M$ is an adjoint of $ad^{k-1}M$ if $ad^{k-2}M$ has an adjoint $ad^{k-1}M$. {\em An adjoint sequence} of $M$ is the sequence $adM^0,adM,ad^2M,\ldots$. Such adjoint sequence may have to stop after finitely many steps with a matroid that fails to admit an adjoint. Based on the fact that a matroid with rank smaller than three always admits an adjoint, we can easily obtain the following result and omit this proof.
\begin{theorem}\label{Finite-Type}
Let $M$ be a simple connected matroid of rank $r\le2$ and size $m$. Then for all integers $k\ge 0$, we have
that $adM^k\cong U_{1,1}$ for $r=1$, and $ad^kM\cong U_{2,m}$ for $r=2$.
\end{theorem}

In order to investigate adjoint sequences of matroids with rank greater than two, we need the next key result.
\begin{lemma}\label{Submatroid}
Let $M$ be a simple matroid. If $M$ has a $2$th adjoint $ad^2M$, then $M$ is a submatroid of $ad^2M$ up to isomorphism.
\end{lemma}
\begin{proof}
Let $\phi_1$ and $\phi_2$ be the adjoint maps of $M$ and $adM$, respectively. According to the definition of adjoint map in \autoref{Sec-2},  the maps $\phi_1$ and $\phi_2$ induce an order-preserving injective map $\phi_2\circ\phi_1$ from $\mathcal{L}(M)$ to $\mathcal{L}(ad^2M)$ sending each element $e\in E(M)$  to $\phi_2\circ\phi_1(e)\in E(ad^2M)$. Taking an independent set $\{e_1,e_2,\ldots,e_k\}$ of $M$,  we know that every $\phi_1(e_i)$ is a hyperplane of $adM$ and $r_{adM}(\bigcap_{j=1}^i\phi_1(e_j)\big)=r(M)-i$ for all $1\le i\le k$ from the properties ${\rm(ii)}$ and ${\rm(iii)}$ in \autoref{Adjoint-Prop}.  Immediately,  we have $\bigcap_{j=1}^{i+1}\phi_1(e_j)\subsetneqq \bigcap_{j=1}^{i}\phi_1(e_j)$ for each $1\le i\le k-1$. From the property ${\rm(v)}$ in \autoref{Adjoint-Prop}, we arrive at that $\{\phi_2\circ\phi_1(e_1),\phi_2\circ\phi_1(e_2),\ldots,\phi_2\circ\phi_1(e_k)\}$ is an independent set of $ad^2M$. So $M$ is isomorphic to a submatroid of $ad^2M$ indeed.
\end{proof}

Below further describes that if a $k$th adjoint $ad^kM$ is isomorphic to the original matroid $M$, then $M$ is isomorphic to its an adjoint $adM$ or $2$th adjoint $ad^2M$.

\begin{lemma}\label{Adjoint-Two}
Let $M$ be a simple matroid. If $M$ has a $k$th adjoint $ad^kM$ for some $k\ge 3$ such that $M\cong ad^kM$, then $M$ is isomorphic to $adM$ or $adM^2$.
\end{lemma}
\begin{proof}
Suppose $k$ is an even number.  According to \autoref{Submatroid}, we know that  $ad^2M$ is a submatroid of $ad^kM$. Noting that $|E(M)|\le|E(ad^2M)|\le|E(ad^kM)|$. It follows from $M\cong ad^kM$ that $M\cong ad^2M$. When $k$ is a odd number, we can prove $M\cong adM$ by the same argument as the above case.
\end{proof}

\autoref{Submatroid} states that $ad^{i}M$ can be embedded naturally into $ad^{i+2}M$ as a submatroid. This leads to an interesting phenomenon that there are two non-decreasing adjoint sequences of matroids, all having the same rank as the original matroid, one beginning with  the original matroid and the other starting with the adjoint matroid. Moreover, \autoref{Adjoint-Two} further implies that if existing distinct non-negative integers $i,j$ satisfies $ad^iM\cong ad^jM$, then the adjoint sequence is cyclic starting with $ad^iM$,  that is, $ad^{i+2k}M\cong ad^iM$ and $ad^{i+2k+1}M\cong ad^{i+1}M$ for all possible integers $k\ge0$. The following result indicates that the cyclic adjoint sequences eventually stabilize at the finite projective geometries.

\begin{theorem}\label{Cyclic-Type}
Let $M$ be a simple connected matroid of rank $r\ge3$. If $M$ has an adjoint sequence $ad^0M,adM,ad^2M,\ldots$ such that  $ad^iM\cong ad^jM$ for some non-negative integers $i<j$, then
\begin{itemize}
\item[{\rm(i)}]  if $r=3$,  the $k$th adjoint $ad^kM$ always exists for all $k\ge i$,  and we have that
\begin{itemize}
\item[{\rm(a)}]  if $j-i$ is odd, then $ad^kM$ is isomorphic to the same finite projective plane for all $k\ge i$;
\item[{\rm(b)}]  if $j-i$ is even,  then $ad^{i+2k}M$ ($ad^{i+2k+1}M$) is isomorphic to the same finite projective plane for all $k\ge 0$ (resp.);
\end{itemize}
\item[{\rm(ii)}]  if $r\ge4$, the $k$th adjoint $ad^kM$ always exists and is isomorphic to the same finite projective geometry $PG(r-1,q)$ for all $k\ge i$.
\end{itemize}
\end{theorem}
\begin{proof}
For $r=3$,  if $j-i$ is odd,  we get $ad^iM\cong ad^{i+1}M$ by \autoref{Adjoint-Two}. This means $\mathcal{L}(ad^{i+1}M)\cong\mathcal{L}(ad^iM)^{op}$ and $\mathcal{L}(ad^iM)^{op}$ is a geometric lattice. Immediately, we have that $ad^iM$ is modular. On the other hand, the uniqueness of an adjoint of a modular matroid in \autoref{Modular-Unique-1} further implies that $ad^{i+1}M$ is the unique adjoint of $ad^iM$ up to isomorphism. It follows from $ad^iM\cong ad^{i+1}M$ that $ad^kM$ always holds and $ad^kM\cong ad^iM$ for all $k>i$. Then $ad^kM$ is isomorphic to a finite projective plane for all $k\ge i$ via \autoref{Modular-1}. If $j-i$ is even,  we acquire $ad^iM\cong ad^{i+2}M$ from \autoref{Adjoint-Two}, which indicates $\mathcal{L}(ad^{i+1}M)\cong\mathcal{L}(ad^iM)^{op}$. Using the same arguments as the above case, we can also arrive at  $ad^{i+2k}M\cong ad^iM$ and $ad^{i+2k+1}M\cong ad^{i+1}M$ for all $k\ge0$. Then we can verify that $ad^{i+2k}M$ ($ad^{i+2k+1}M$)  is isomorphic to a finite projective plane by \autoref{Modular-1} (resp.).

For $r\ge4$,  as an application of  \autoref{Adjoint-Two}, we have that $ad^iM$ is isomorphic to $ad^{i+1}M$ or $ad^{i+2}M$. It implies $\mathcal{L}(ad^{i+1}M)\cong\mathcal{L}(ad^iM)^{op}$ in the both potential cases. Likewise, $ad^iM$ is modular.  Combing with \autoref{Modular-Adjoint}, we arrive at that $ad^kM$ always exists and $ad^kM\cong ad^iM$  for all $k>i$ via the uniqueness of the adjoint of a modular matroid. So for all $k\ge i$, $ad^kM$ is isomorphic to a finite projective geometry $PG(r-1,q)$ from \autoref{Adjoint-Remark}. The proof is completed.
\end{proof}

To classify the infinite non-repeating adjoint sequences, we first introduce the direct limit of a directed system associated to matroids.  Let $ M=(E,\mathcal{I})$ and $ M'=(E',\mathcal{I}')$ be two matroids. An injection $\iota:  M\hookrightarrow M'$ is called an {\em embedding} if the image $\iota(M):=\b(\iota(E),\iota(\mathcal{I})\b)$ of $\iota$ is a submatroid of  $M'$. Let $\mathcal{M}$ be a category of all matroids including finite and infinite matroids. More information for the infinite matroids may refer to \cite{Oxley1978,Oxley}. Let $\{M_i\mid i\in\mathbb{N}\}$ be a family of matroids in $\mathcal{M}$. For each pair $i,j\in \mathbb{N}$ with $i\le j$,  assume given an embedding map $f_{ij}: M_i\hookrightarrow M_j$ such that, whenever $i\le j\le k$ in $\mathbb{N}$, we have
\[
f_{\sst jk}\circ f_{\sst ij}=f_{ik}\quad \And \quad f_{ii}={\rm id},
\]
where id denotes the identity mapping. Such triple $\big(\mathbb{N},\{ M_i\},\{f_{ij}\}\big)$ is called a {\em directed system} in $\mathcal{M}$.
\begin{definition}\label{Direct-Limit}
{\em
Let $\mathcal{M}$ be a category of all matroids and $\big(\mathbb{N},\{M_i\},\{f_{ij}\}\big)$ a directed system in $\mathcal{M}$.
An element $M\in \mathcal{M}$ is called a {\em direct limit} of this system if there exists
an embedding map $f_i:M_i\hookrightarrow M$ for each $i\in\mathbb{N}$ with the following properties:
\begin{itemize}
  \item [\rm{(i)}] $f_i=f_j\circ f_{ij}$ for any integer $i\le j$ in $\mathbb{N}$.
  \item [\rm{(ii)}] Given an element $N\in\mathcal{M}$ and an embedding $g_i: M_i\hookrightarrow N$ such that $g_i=g_j\circ f_{ij}$ for any integer $i\le j$ in $\mathbb{N}$, there exists unique embedding $g: M\hookrightarrow N$ such that $g_i=g\circ f_i$.
\end{itemize}
Such direct limit, write $M=\lim\limits_{\begin{subarray}{c}\rightarrow\\i\end{subarray}} M_i$.
}
\end{definition}

Let $M$ be a simple connected matroid. Suppose $M$ has an infinite adjoint sequence $ad^0M,adM,ad^2M,\ldots$.  Let $\phi_k^2$ be an adjoint map from $ad^{2k}M$ to $ad^{2k+1}M$, and $\phi_k^1$ be an adjoint map from $ad^{2k+1}M$ to $ad^{2(k+1)}M$ for all integers $k\in\mathbb{N}$. Recall from \autoref{Submatroid} that for any integers $i\le j$ in $\mathbb{N}$, $ad^{2i}M$ is a submatroid of $ad^{2j}M$, which yields a natural  embedding map
\[
\phi_{ij}:ad^{2i}M \hookrightarrow ad^{2j}M\quad \mbox{ such that } \quad\phi_{ij}=\phi_{j-1}^1\circ\phi_{j-1}^2\cdots\phi_i^1\circ\phi_i^2.
\]
It is clear that for any integers $i\le j\le k$ in $\mathbb{N}$, we have $\phi_{ik}=\phi_{jk}\circ\phi_{ij}$ and $\phi_{ii}={\rm id}$. Therefore, the triple $\big(\mathbb{N},\{ad^{2i}M\},\{\phi_{ij}\}\big)$ is a directed system in $\mathcal{M}$.

Next we shall construct a matroid  $\bar{M}$ (may be infinite) associated to the even adjoint sequence $ad^0M,ad^2M,\ldots$,  which is turned out to be a direct limit of  the directed system $\big(\mathbb{N},\{ad^{2i}M\},\{\phi_{ij}\}\big)$ later. Let $E_{\infty}=\bigcup_{i=0}^\infty E(ad^{2i}M)$. Define an equivalence relation $\sim$ on $E_\infty$ such that for any members $e,f\in E_\infty$, $e\sim f$ if and only if existing integers $i\le j$ in $\mathbb{N}$ satisfy that $e\in E(ad^{2i}M)$, $f\in E(ad^{2j}M)$ and $\phi_{ij}(e)=f$. Given an element $e\in E_\infty$, let $\bar{e}$ be an equivalence class of elements of $E_\infty$ containing $e$, that is $\bar{e}:=\{f\mid f\sim e, f\in E_\infty\}$. Let $\bar{E}$ be the set of all equivalence classes of $E_\infty$, i.e., $\bar{E}:=E_\infty/\sim=\b\{\bar{e}\mid e\in E_\infty\b\}$. Let $\bar{\mathcal{I}}$ be a collection of subsets of $\bar{E}$ such that for any subset $I=\{\bar e_1,\bar e_2,\ldots,\bar e_j\}$ of $\bar E$,  assume $e_k\in ad^{2i_k}M$  for each $k=1,2,\ldots,j$, and $i_1\le i_2\le \cdots\le i_j$,  $I\in \bar{\mathcal{I}}$ if and only if $\big\{\phi_{i_1i_j}(e_1),\phi_{i_2i_j}(e_2),\ldots,\phi_{i_ji_j}(e_j)\big\}$ is an independent set of $ad^{2i_j}M$. In this case, obviously, the set $\big\{\phi_{i_1l}(e_1),\phi_{i_2l}(e_2),\ldots,\phi_{i_jl}(e_j)\big\}$ is independent in $ad^{2l}M$  whenever $l\ge i_j$.

\begin{lemma}\label{Matroid-Projective-1}
Let $M$ be a simple connected matroid of rank $r$. If $M$ has an infinite adjoint sequence $ad^0M,adM,ad^2M,\ldots$, then $\bar{M}=(\bar{E},\bar{\mathcal{I}})$ is a simple connected matroid of rank $r$.
\end{lemma}
\begin{proof}
We first shall prove that $\bar{M}=(\bar{E},\bar{I})$  is a matroid. Obviously, $\emptyset\in\bar{\mathcal{I}}$. For any member $I$ of $\bar{\mathcal{I}}$, the construction of $\bar{\mathcal{I}}$ guarantees that all subsets of $I$ are also the members of $\bar{\mathcal{I}}$. Given two elements $I_1=\{\bar e_1,\bar e_2,\ldots,\bar e_k\}$ and $I_2=\{\bar f_1,\bar f_2,\ldots,\bar f_l\}$ with $k<l$. We may assume that  $e_m\in ad^{2i_m}M$ ($m=1,2,\ldots,k$) with $i_1\le i_2\le \cdots\le i_k$,  $f_n\in ad^{2j_n}M$ ($n=1,2,\ldots,l$) with $j_1\le j_2\le \cdots\le j_l$ and $i_k\le j_l$. According to the definition of $\bar{\mathcal{I}}$, we have that $I_1'=\big\{\phi_{i_1j_l}(e_1),\phi_{i_2j_l}(e_2),\ldots,\phi_{i_kj_l}(e_k)\big\}$ and $I_2'=\big\{\phi_{j_1j_l}(f_1),\phi_{j_2j_l}(f_2),\ldots,\phi_{j_lj_l}(f_l)\big\}$ are the independent sets of  $ad^{2j_l}M$.  Since $ad^{2j_l}M$ is a matroid, the both independent sets $I_1'$ and $I_2'$ satisfy  the independence augmentation property in $ad^{2j_l}M$. This implies that $I_1$ and $I_2$ meet the independence augmentation property in $\bar M$ as well.  Hence,  the ordered pair $(\bar{E},\bar{\mathcal{I}})$ satisfies the independence axiom of matroid. Namely, $\bar{M}$ is a matroid. Moreover, recall the definition of $\bar{\mathcal{I}}$ again, we can easily obtain that $|I|\le r$ for any member $I\in\bar{\mathcal{I}}$,  and $\{\bar e_1,\bar e_2,\ldots,\bar e_r\}$ is an independent set of $\bar M$ for a basis $\{e_1,e_2,\ldots,e_r\}$ of $M$. So, the rank of $\bar M$ equals $r$.  Obviously, $\bar{M}$ is simple since each matroid $ad^{2i}M$ is simple. In addition,  taking any members $\bar{e}$ and $\bar{f}$ of $\bar{E}$, we may assume that $e,f\in E(ad^{2j}M)$ for some non-negative integer $j$. Then we obtain from \autoref{Adjoint-Con2} that $ad^{2j}M$ is connected since $M$ is connected. It follows that $ad^{2j}M$ has a circuit $C$ containing $e,f$.  Let $C=\{e,f,g_1,\ldots,g_i\}$ and $\bar{C}=\{\bar{e},\bar{f},\bar{g}_1,\ldots,\bar{g}_i\}$.  Immediately, the construction of $\bar{M}$ implies that $\bar{C}$ is a circuit of $\bar{M}$. Thus $\bar{M}$ is connected. This proof is completed.
\end{proof}

\begin{proposition}\label{Matroid-Projective-2}
Let $M$ be a simple connected matroid. If $M$ has an infinite adjoint sequence $ad^0M,adM,ad^2M,\ldots$, then $\bar{M}$ is a direct limit of the directed system $\big(\mathbb{N},\{ad^{2i}M\},\{\phi_{ij}\}\big)$, namely, $\bar{M}=\lim\limits_{\begin{subarray}{c}\to\\i\end{subarray}}ad^{2i}M$.
\end{proposition}
\begin{proof}
\autoref{Matroid-Projective-1} states that $\bar{M}$ is a matroid. We define a map $\phi_i:ad^{2i}M\to\bar{M}$ sending $e\in E(ad^{2i}M)$ to its equivalence class $\bar{e}\in\bar{E}$. Obviously, $\phi_i$ is injective. Note from the construction of $\bar{\mathcal{I}}$ that for any subset $I=\{e_1,\ldots,e_k\}$ of $E(ad^{2i}M)$, if $I$ is an independent set of $ad^{2i}M$, then $\{\bar{e}_1,\ldots,\bar{e}_k\}$ is independent in $\bar{M}$. Namely, the image $\phi_i(ad^{2i}M)$ of $\phi_i$ is a submatroid of $\bar{M}$. So, $\phi_i$ is an embedding map.  Additionally, for any integers $i\le j$ in $\mathbb{N}$ and $e\in E(ad^{2i}M)$, we have $\phi_i(e)=\phi_j\circ\phi_{ij}(e)=\bar{e}$ since $\phi_{ij}(e)\sim e$. Thus, we has obtained $\phi_i=\phi_j\circ\phi_{ij}$ for any integers $i\le j$ in $\mathbb{N}$. Namely, $\phi_i$ meets the property ${\rm (i)}$ in \autoref{Direct-Limit}.

Next we will show that $\phi_i$ satisfies the property ${\rm (ii)}$ in \autoref{Direct-Limit}. Given a matroid $N\in\mathcal{M}$ and an embedding map $\psi_i:ad^{2i}M\to N$ such that $\psi_i=\psi_j\circ\phi_{ij}$ for any integers $i\le j$ in $\mathbb{N}$. We need to show that there is a unique embedding map $\psi:\bar{M}\to N$ for which $\psi_i=\psi\circ\phi_i$. Define a map
$\psi:\bar{M}\to N$ such that $\psi(\bar{e})=\psi_i(e)$ if $\bar{e}=\phi_i(e)$ for some $e\in E(ad^{2i}M)$. The relation $\psi_i=\psi_j\circ\phi_{ij}$ implies that $\psi$ is well defined. Firstly, we shall prove the injectivity of $\psi$. Given two distinct members $\bar{e}$ and $\bar{f}$ of $\bar{E}$,  we may assume that $\psi(\bar{e})=\psi_i(e)$ for $e\in E(ad^{2i}M)$ and $\psi(\bar{f})=\psi_j(f)$ for $f\in E(ad^{2j}M)$  with $i\le j$ in $\mathbb{N}$. Since $\bar{e}\ne \bar{f}$, $\phi_{ij}(e)\ne f$ in $ad^{2j}M$. Then the injectivity of $\psi_j$ means that $\psi_i(e)=\psi_j\circ\phi_{ij}(e)\ne \psi_j(f)$. Namely, $\psi$ is injective. Secondly, we will verify that the image $\psi(\bar{M})$ of $\psi$ is a submatroid of $N$. This problem can reduce to showing that for a fixed independent set $\bar{I}=\{\bar{e}_1,\bar{e}_2,\ldots,\bar{e}_j\}$ of $\bar{M}$,  the set $\psi(\bar{I}):=\{\psi(\bar{e}_1),\psi(\bar{e}_2),\ldots,\psi(\bar{e}_j)\}$ is independent in $N$. Suppose $\psi(\bar{e}_k)=\psi_{i_k}(e_k)$ for some $e_k\in E(ad^{2i_k}M)$ and  $i_1\le i_2\cdots\le i_j$. According to the construction of $\bar{M}$, we arrive at that $\{\phi_{i_1i_j}(e_1),\phi_{i_2i_j}(e_2),\ldots,\phi_{i_ji_j}(e_j)\}$ is an independent set of $ad^{2i_j}M$ since $\bar{I}$ is independent in  $\bar{M}$. Noticing that $\psi(\bar{e}_k)=\psi_{i_j}\circ\phi_{i_ki_j}(e_k)$ for $k=1,2,\ldots,j$. Immediately, we can obtain that  $\psi(\bar{I})$ is independent in $N$ since $\psi_{i_j}$ is an embedding map from $ad^{2i_j}M$ to $N$. Up to now, we have verified that $\psi$ is an embedding map and $\psi_i=\psi\circ\phi_i$.  Note that the two embedding maps $\psi_i$ and $\phi_i$ completely determine the map $\psi$, that is, $\psi$ is unique. Hence, $\bar{M}$ is a directed limit of $\big(\mathbb{N},\{ad^{2i}M\},\{\phi_{ij}\}\big)$. We complete the proof.
\end{proof}

\autoref{Matroid-Projective-2} states that $\lim\limits_{\begin{subarray}{c}\to\\i\end{subarray}}ad^{2i}M$ is a matroid. Below further shows that $\lim\limits_{\begin{subarray}{c}\to\\i\end{subarray}}ad^{2i}M$ is a projective geometry, which is implicitly contained in \cite[Hint of Exercise 7.17]{Bjorner1999}. We shall omit its straightforward proof.
\begin{proposition}\label{Convergent-1}
Let $M$ be a simple connected matroid of rank $r\ge3$. If $M$ has an infinite adjoint sequence $ad^0M,adM,ad^2M,\ldots$,  then  $\lim\limits_{\begin{subarray}{c}\to\\i\end{subarray}}ad^{2i}M$ is a projective geometry.
\end{proposition}

Analogous to $\lim\limits_{\begin{subarray}{c}\to\\i\end{subarray}}ad^{2i}M$, we can define $\lim\limits_{\begin{subarray}{c}\to\\i\end{subarray}}ad^{2i+1}M$ by making a minor change. Additionally, using the same arguments as in \autoref{Matroid-Projective-1}, \autoref{Matroid-Projective-2} and \autoref{Convergent-1}, we can verify that if a connected matroid $M$ with rank greater than two has an infinite adjoint sequence $ad^0M,adM,ad^2M,\ldots$,  then  $\lim\limits_{\begin{subarray}{c}\to\\i\end{subarray}}ad^{2i+1}M$ is a projective geometry. We are now ready to characterize the infinite non-repeating adjoint sequences by infinite projective geometries. \begin{theorem}\label{Convergent-Type}
Let $M$ be a simple connected matroid of rank $r\ge3$. If $M$ has an infinite adjoint sequence $ad^0M,adM,ad^2M,\ldots$ such that $ad^iM\ncong ad^jM$ for any non-negative integers $i<j$, then
\begin{itemize}
\item[{\rm(i)}]  if $r=3$, $\lim\limits_{\begin{subarray}{c}\to\\i\end{subarray}}ad^{2i} M$ ($\lim\limits_{\begin{subarray}{c}\to\\i\end{subarray}}ad^{2i+1} M$) is  an infinite  projective plane.
\item[{\rm(ii)}]  if $r\ge4$,  $\lim\limits_{\begin{subarray}{c}\to\\i\end{subarray}}ad^{2i} M$ and $\lim\limits_{\begin{subarray}{c}\to\\i\end{subarray}}ad^{2i+1} M$ are isomorphic to the same infinite projective geometry $PG(r-1,\mathbb{F})$ for some infinite field $\mathbb{F}$.
\end{itemize}
\end{theorem}
\begin{proof}
For $r\ge 4$, according to \autoref{Veblen-Young} and \autoref{Convergent-1},  we obtain that  $\lim\limits_{\begin{subarray}{c}\to\\i\end{subarray}}ad^{2i} M$ is isomorphic to the projective geometry  $PG(r-1,\mathbb{F})$  for some field $\mathbb{F}$. Suppose $\mathbb{F}$ is a finite field $GF(q)$ with $q$ elements. Then $PG(r-1,\mathbb{F})$ contains finitely many distinct submatroids (up to isomorphism). Notice from \autoref{Submatroid} that $ad^{2i}M$ can be viewed as a submatroid of $PG(r-1,\mathbb{F})$ for all $i\ge 0$. The preceding arguments make $ad^iM\cong ad^jM$ for some non-negative integers $i<j$. This contradicts the assumption that
$ad^iM\ncong ad^jM$ for all $0\le i<j$. So $\mathbb{F}$ is an infinite field. Namely, $PG(r-1,\mathbb{F})$ is an infinite projective geometry.  This implies that each $ad^{2i}M$ can be representable over $\mathbb{F}$.  Recall from  \cite[Lemma 2.8]{Bixby-Coullard1988} that if $M$ has an adjoint $adM$ and $M$ is representable over a field $\mathbb{F}$, then $adM$ is isomorphic to some type I adjoint of $M$. Hence, every $ad^{2i+1}M$ is representable over the field $\mathbb{F}$. Similar to the argument of  $\lim\limits_{\begin{subarray}{c}\to\\i\end{subarray}}ad^{2i} M$,  we can arrive at that $\lim\limits_{\begin{subarray}{c}\to\\i\end{subarray}}ad^{2i+1} M$ is also isomorphic to the same infinite projective geometry  $PG(r-1,\mathbb{F})$.

For $r=3$,  recall from \autoref{Convergent-1} that  $\lim\limits_{\begin{subarray}{c}\to\\i\end{subarray}}ad^{2i} M$ ($\lim\limits_{\begin{subarray}{c}\to\\i\end{subarray}}ad^{2i+1} M$) is a projective plane $(P,L,\iota)$.  Using the same argument as in the proof of the case $r\ge 4$, we can also obtain $(P,L,\iota)$ is an infinite projective plane. We complete the proof.
\end{proof}

Next we are ready to handle the classification problem of adjoint sequences for arbitrary matroids. Let us recall \autoref{Adjoint-Con2} that an adjoint of a direct sum of two matroids is the direct sum of the adjoints of these matroids. Immediately, the following result is now a direct consequence of \autoref{Finite-Type}, \autoref{Cyclic-Type} and \autoref{Convergent-Type}.
\begin{corollary}
Let $M$ be a simple matroid and write as a direct sum of  its connected components $M_1,\ldots, M_n$. Then
\begin{itemize}
  \item [{\rm (i)}] if the rank of each component of $M$ is no more than two, then the $k$th adjoint $ad^kM$ always exists and $ad^kM\cong M$ for all $k\ge0$;
  \item [{\rm (ii)}] if $M$ has an adjoint sequence $ad^0M,adM,\ldots$ such that  $ad^iM\cong ad^jM$ for some non-negative integers $i<j$,
  then the $k$th adjoint $ad^kM$ always exists and each connected component of $ad^kM$ is isomorphic to the free matroid $U_{1,1}$ or a finite projective geometry for all $k\ge i$;
  \item [{\rm (iii)}] if $M$ has an infinite adjoint sequence $ad^0M,adM,ad^2M,\ldots$ such that $ad^iM\ncong ad^jM$ for any non-negative integers $i<j$, then each connected component of the direct limit $\lim\limits_{\begin{subarray}{c}\to\\i\end{subarray}}ad^{2i} M$ ($\lim\limits_{\begin{subarray}{c}\to\\i\end{subarray}}ad^{2i+1} M$) is the free matroid $U_{1,1}$ or a projective geometry,  and $M$ has at least one component $M_{k}$ such that $\lim\limits_{\begin{subarray}{c}\to\\i\end{subarray}}ad^{2i} M_{k}$ is the infinite projective geometry.
\end{itemize}
\end{corollary}

We close this section by discussing the link between matroid representability and infinite adjoint sequences.  A basic question in matroid theory is how to find whether a matroid is representable. As far as we know, this question is very challenging and still open. Noting that if a matroid $M$ is representable over some field, then it has an infinite type I adjoint sequence. Conversely, suppose $M$ with rank greater than three has an infinite adjoint sequence, applying \autoref{Convergent-1} to this adjoint sequence, we arrive at that $M$ can be viewed as a submatroid of a projective geometry. Immediately,  the Coordinatization Theorem in \autoref{Veblen-Young} means that $M$ is representable. We conclude the result as follows.
\begin{corollary}
Let $M$ be a simple connected matroid with rank greater than three. Then $M$ is representable if and only if $M$ has an infinite adjoint sequence.
\end{corollary}
\section{Type I adjoint}\label{Sec-6}
\subsection{Type I adjoint sequences}
To our knowledge, a matroid may fail to admit an adjoint. However, if a matroid is representable, it has always a type I adjoint.  It is worth noting that the type I adjoint of a matroid depends on its a representation. As a byproduct of \autoref{Sec-5}, associated with a fixed $\mathbb{F}$-represented matroid, this section will simplify the classifications of the type I adjoint sequences into two types: finite and convergent. For a field $\mathbb{F}$,  let $M$ be an $\mathbb{F}$-representable matroid on ground set $E(M)=\{e_1,e_2,\ldots,e_m\}$, let $\varphi: E(M)\to\mathbb{F}^n$ be a representation of $M$. The matrix $A$ whose columns are the vectors $\varphi(e_1),\varphi(e_2),\ldots,\varphi(e_m)$ is the matrix corresponding to $\varphi$. Namely, $M$ is $M[A]$. The matrix $A$ is known as an $\mathbb{F}$-representation of $M$. Moreover, the pair $(M,\varphi)$, or equivalently the pair $(M,A)$, denotes an $\mathbb{F}$-represented matroid.

Let $M$ be an $\mathbb{F}$-represented matroid on ground set $E(M)=\{e_1,e_2,\ldots,e_m\}$ of rank $r$ with the representation $\varphi: E(M)\to\mathbb{F}^r$. For each hyperplane $H\in \mathcal{H}(M)$,  all the vectors $\varphi(e_i)$ with $e_i\in H$ automatically generate a hyperplane ${\rm span}(H)$ in $\Bbb{F}^r$.  Let ${\bm h}_H$ be the normal vector of ${\rm span}(H)$ in $\Bbb{F}^r$ when $H\in \mathcal{H}(M)$. This yields an $\mathbb{F}$-represented matroid $(\sigma M,\sigma\varphi)$ with ground set $\mathcal{H}(M)$ such that
$(\sigma\varphi)(H)={\bm h}_H$ for all hyperplanes $H$ in $\mathcal{H}(M)$. The {\em type I adjoint} $(\sigma M,\sigma\varphi)$ of $(M,\varphi)$ is defined as
\begin{equation*}
(\sigma M,\sigma\varphi):=M\b[{\bm h_H}\mid H\in \mathcal{H}(M)\b].
\end{equation*}
We shall frequently write $\sigma M$ for $(\sigma M,\sigma\varphi)$. Moreover, let $(\sigma^0M,\sigma^0\varphi)=(M,\varphi)$. It is also possible to repeat the procedure of taking the type I adjoint, for any positive integer $k$, the {\em $k$th  type I adjoint} $(\sigma^kM,\sigma^k\varphi)$ of $M$ is the type I adjoint of $(\sigma^{k-1}M,\sigma^{k-1}\varphi)$. By the construction of the type I adjoint, $(M,\varphi)$ has an infinite type I adjoint sequence $(\sigma^0M,\sigma^0\varphi)$, $(\sigma^1M,\sigma^1\varphi)$, $(\sigma^2M,\sigma^2\varphi),\ldots$, which is referred as Crapo sequence in \cite{Kung2020} as well.

To obtain a more precise classification of the type I adjoint sequences, let us first introduce Desargues' theorem. Desargues' theorem is one of the most fundamental and beautiful results in projective geometry. Desargues' theorem states that given three distinct lines $a_1b_1, a_2b_2$ and $a_3b_3$ in $PG(2,\mathbb{F})$, if the three lines meet at the point $o$, then the points $c_1=a_1a_2\cap b_1b_2$, $c_2=a_1a_3\cap b_1b_3$ and $c_3=a_2a_3\cap b_2b_3$ are collinear. Below describes the other beautiful property that if a projective plane satisfies Desargues' theorem, then it is isomorphic to a projective plane $PG(2,\mathbb{F})$ obtained from a vector space $\mathbb{F}^3$.

\begin{theorem}{\rm\cite{Dembowski1968}}\label{Projective-Plane-0}
Let $(P,L,\iota)$ be a projective plane. Then $(P,L,\iota)$ is isomorphic to $ PG(2,\mathbb{F})$ for a filed $\mathbb{F}$ if and only if Desargues' theorem holds. In particular, if $(P,L,\iota)$ is a finite projective plane, then $(P,L,\iota)$ is isomorphic to $PG(2,q)$ for some prime power $q$ if and only if Desargues' theorem holds.
\end{theorem}

Now we are ready to present that  if an $\mathbb{F}$-represented matroid $(M,\varphi)$ is a projective plane, then Desargues' theorem holds.
\begin{lemma}\label{Projective-Plane}
Let $(M,\varphi)$ be an $\mathbb{F}$-represented matroid with the representation $\varphi: E(M)\to\mathbb{F}^3$. If $M$ is a projective plane, then $M\cong PG(2,\mathbb{K})$ for some subfield $\mathbb{K}$ of $\mathbb{F}$.
\end{lemma}
\begin{proof}
To prove the lemma, from the \autoref{Projective-Plane-0}, it is sufficient to show that the Desargues' theorem holds in $M$. Since $M$ is an $\mathbb{F}$-represented matroid, then $M\cong PG(2,\mathbb{F})|\varphi(E)$, where $\varphi(E)$ is a finite set of $PG(2,\mathbb{F})$. Then we may assume that $M=PG(2,\mathbb{F})|\varphi(E)$. So three points in $M$ are collinear if and only if their representing vectors are linearly dependent.  Let $a_1b_1, a_2b_2$ and $a_3b_3$ be three lines in $M$ all meeting at the point $o$ with $a_i,b_i\in\varphi(E)$ for $i=1,2,3$ and $o\in\varphi(E)$. Since $a_i\ne b_i$ for $i=1,2,3$, then the vectors $a_i$ and $b_i$ form a basis of $2$-dimensional subspace corresponding to the line $a_ib_i$ for $i=1,2,3$. Since the point $o$ is incident with each of lines $a_ib_i$, then there exist $\lambda_i,\mu_i\in\mathbb{F}$ such that
\begin{equation*}\label{Linearly-Dependent}
o=\lambda_1a_1+\mu_1b_1,\;o=\lambda_2a_2+\mu_2b_2,\;o=\lambda_3a_3+\mu_3b_3
\end{equation*}
hold. From the above equation, we have
\begin{equation*}
c_1=\lambda_1a_1-\lambda_2a_2=\mu_2b_2-\mu_1b_1,c_2=\lambda_2a_2-\lambda_3a_3=\mu_3b_3-\mu_2b_2,
c_3=\lambda_3a_3-\lambda_1a_1=\mu_1b_1-\mu_3b_3.
\end{equation*}
Note that $c_1$ is the linear combination of $a_1$ and $a_2$, and $c_1$ is also the linear combination of $b_1$ and $b_2$. It implies $c_1$ lies on the lines $a_1a_2$ and $b_1b_2$. Since $M$ is a projective plane, then $a_1a_2$ and $b_1b_2$ has the unique common point, which implies $c_1\in\varphi(E)$ and $c_1\in a_1a_2\cap b_1b_2$. In the same way, we get $c_2\in\varphi(E)$ and $c_2\in a_2a_3\cap b_2b_3$ and $c_3\in\varphi(E)$ and $c_3\in a_1a_3\cap b_1b_3$. The equality
\[
c_1+c_2+c_3=(\lambda_1a_1-\lambda_2a_2)+(\lambda_2a_2-\lambda_3a_3)+(\lambda_3a_3-\lambda_1a_1)=0
\]
means that $c_1,c_2$ and $c_3$ are linearly dependent, i.e., $c_1,c_2$ and $c_3$ are collinear. Hence the lemma is proved.
\end{proof}

It is clear that \autoref{Projective-Plane} holds for infinite matroids as well. Then, combining with \autoref{Projective-Plane}, applying \autoref{Finite-Type}, \autoref{Cyclic-Type} and \autoref{Convergent-Type} to the type I adjoint sequences, we can obtain the classifications of the type I adjoint sequences directly, which is implicit in the work of Kung \cite{Kung2020}.
\begin{theorem}
Let $(M,\varphi)$ be a simple connected $\mathbb{F}$-represented matroid of rank $r$ with the representation $\varphi: E(M)\to\mathbb{F}^r$, then
\begin{itemize}
\item[{\rm(i)}] when $r=1,2$,  then $\sigma^iM\cong U_{1,1}\And U_{2,m}$ for all $i\ge 0$, respectively;
\item[{\rm(ii)}] when $r\ge3$, if $\sigma^iM\cong \sigma^jM$ for some non-negative integers $i<j$, then $\mathbb{F}$ is a finite field and $\sigma^kM\cong PG(r-1,q)$ for all $k\ge i$, where $GF(q)$ with $q$ elements is a subfield of $\mathbb{F}$;
\item[{\rm(iii)}] when $r\ge3$,  if $\sigma^i M\ncong\sigma^j M$ for any non-negative integers $i<j$, then $\mathbb{F}$ is an infinite field and $\lim\limits_{\begin{subarray}{c}\rightarrow\\i\end{subarray}}\sigma^{2i} M$ and $\lim\limits_{\begin{subarray}{c}\rightarrow\\i\end{subarray}}\sigma^{2i+1} M$ are isomorphic to the same infinite projective geometry  $PG(r-1,\mathbb{K})$,
      where $\mathbb{K}$ is an infinite subfield of $\mathbb{F}$.
\end{itemize}
\end{theorem}

\subsection{Duality}
When $M$ is a vector matroid, Bixby and Coullard \cite{Bixby-Coullard1988} constructed  an adjoint of $ M$ in two equivalent ways: one is from cocircuits of $M$, the other is from hyperplane flats of $ M$.  From their constructions, we realize that there is a duality phenomenon between the adjoint and derived matroid of vector matroids. Such phenomenon has appeared in \cite{Falk1994} as well, when Falk studied the discriminantal arrangements for a general position configurations.

For a field $\mathbb{F}$, given an $\mathbb{F}$-represented matroid $(M,A):=M[A]$ of rank $r$ on ground set $E(M)=\{e_1,e_2,\ldots,e_m\}$, where the columns of  the matrix $A\in \Bbb{F}^{r\times m}$ are labelled, in order, $e_1,e_2,\ldots,e_m$. For every circuit $C\in\mathcal{C}( M)$, there is a unique vector ${\bm c}_{C}= (c_1,c_2,\dots,c_m)$ (up to a non-zero scalar multiple) in ${\mathbb F}^m$ such that  $\sum_{i = 1}^m c_i A_{e_i}= {\bm0}$, where  $c_i\neq 0$ if and only if $e_i \in C$, called the {\em circuit vector} of $C$. The {\em Oxley-Wang  derived matroid}  $\delta_{OW}M$ is defined as
\begin{equation}\label{derived-matroid}
\delta_{OW} M:= M[{\bm c}_{C}\mid C\in\mathcal{C}( M)].
\end{equation}
Let $A'\in\mathbb{F}^{(m-r)\times m}$ be a matrix such that the set of all row vectors of $A'$ is a basis of the solution space of $A{\bm x}={\bm 0}$. It is clear from matroid theory \cite[\S 2.2]{Oxley} that the dual matroid $M^*$ of $M$ can be naturally represented by column vectors of $A'$, i.e., $(M^*,A'):= M[A']$, where the columns of the matrix $A'$ are also labelled, in order, $e_1,e_2,\ldots,e_m$. Accordingly, we can define the Oxley-Wang derived matroid $\delta_{OW} M$ and the type I adjoint $\sigma M^*$ via $A$ and $A'$. Below will present a detailed proof of the duality relation in \eqref{Adjoint-Derived}, which is an illuminating relation for future studies associated to the combinatorial derived matroid defined in \cite{Freij2023}.
\begin{proposition}\label{Duality}If $(M,A)$ and $(M^*,A')$ are defined as above, then
\[\delta_{OW}M\cong\sigma M^*.\]
\end{proposition}
\begin{proof}
Recall from \eqref{Adjoint-I} and \eqref{derived-matroid} that
\[
\delta_{OW} M= M\b[{\bm c}_{C}\mid C\in\mathcal{C}( M)\b],\quad
\sigma M^*= M\b[{\bm h_{H}}\mid H\in \mathcal{H}( M^*)\b].
\]
Note from the matroid theory that $C$ is a circuit of $ M$ if and only if $C^*:=E(M)\setminus C$ is a hyperplane of $M^*$, see \cite[Proposition 2.1.6]{Oxley}. It suffices to show that for any circuits $C_1,C_2,\ldots,C_k\in\mathcal{C}( M)$, the both sets $\{{\bm c}_{C_1},{\bm c}_{C_2},\ldots,{\bm c}_{C_k}\}$ and $\{{\bm h_{C^*_1}},{\bm h_{C^*_2}},\ldots,{\bm h_{C^*_k}}\}$ have the same rank. Given a circuit $C\in \mathcal{C}( M)$ with the circuit vector ${\bm c}_{C}=(c_1,c_2,\ldots,c_m)$, i.e., $A{\bm c}_{C}^{\sst\rm T}={\bm0}$ and $c_i\ne 0$ if and only if $e_i\in C$. It follows from the definition that ${\bm c}_{C}$ can be written uniquely as a linear combination of row vectors of $A'$, say ${\bm c}_{C}={\bm h}A'$, where ${\bm h} \ne {\bm 0}$ obviously. Since $c_i=0$ if and only if $e_i\in C^*$, we have ${\bm h}A_{e_i}'=0$ for all $e_i\in C^*$.  Namely, ${\bm h}$ is the normal vector of the hyperplane ${\rm span}\{A'_{e_i}:e_i\in C^*\}$ in $\mathbb{F}^{m-r}$ that is spanned by all columns of $A'$ labelled by $e_i\in C^*$ . So by the definition of ${\bm h}_{C^*}$ we may assume ${\bm h}_{C^*}={\bm h}$. From above arguments, we have obtained for any circuits $C_1,C_2,\ldots,C_k\in\mathcal{C}( M)$,
\begin{eqnarray*}
&{\sst\left[\begin{array}{ccc}
{\bm c}_{\sst C_1}\\
\vdots\\
{\bm c}_{\sst C_k}
\end{array}\right]}
={\sst\left[\begin{array}{ccc}
{\bm h_{\sst C^*_1}}\\
\vdots\\
{\bm h_{\sst C^*_k}}
\end{array}\right]}A'.
\end{eqnarray*}
Since row vectors of $A'$ are of full rank, it follows that the rank of
$\{{\bm c}_{C_1},{\bm c}_{C_2},\ldots,{\bm c}_{C_k}\}$ equals the rank of $\{{\bm h_{C^*_1}},{\bm h_{C^*_2}},\ldots,{\bm h_{C^*_k}}\}$  and completes the proof.
\end{proof}

In general, $\sigma M$ may depend on the $\mathbb{F}$-representation of $M$. Oxley and Wang \cite[Theorem 8]{Oxley-Wang2019} showed that the Oxley-Wang derived matroid $\delta_{OW}M$ of an $\mathbb{F}$-represented matroid $(M,A)$ does not depend on the $\mathbb{F}$-representation $A$ if and only if $\mathbb{F}$ is $GF(2)$ or $GF(3)$. An immediate result comes from this by \autoref{Duality}.
\begin{corollary}
Let $\mathbb{F}$ be a field. Then, for all $\mathbb{F}$-represented matroids $(M,A)$ of rank $r$ and size $m$, the type I adjoint $\sigma M$ does not depend on the $\mathbb{F}$-representation $A$ if and only if $\mathbb{F}$ is $GF(2)$ or $GF(3)$, where the matrix $A\in\mathbb{F}^{r\times m}$.
\end{corollary}

\section{Further research}\label{Sec-7}
In this section, of utmost initial motivation to us is to study the relevant questions appearing in \cite{Freij2023} associated with combinatorial derived matroids. Recently, as a generalization of the Oxley-Wang derived matroid, Freij-Hollanti, Jurrius and Kuznetsova \cite{Freij2023} constructed the combinatorial derived matroid for arbitrary matroids. Unlike the adjoint of matroids, this construction guarantees that an arbitrary matroid has always a  derived matroid. In particular, they showed that by choosing an appropriate representation for a representable matroid, its Oxley-Wang derived matroid may coincide with the combinatorial derived matroid. Moreover, \cite{Freij2023} also stated some close connections between the combinatorial derived matroid and the adjoint of matroids, and further proposed a number of particularly interesting and insightful questions. Before proceeding further, we first introduce the definition of the combinatorial derived matroid.
\begin{definition}
{\rm
Let $M$ be a matroid and the collection
\[
\mathcal{D}_0:=\Big\{D\subseteq\mathcal{C}(M)\mid |D|>\big|\bigcup_{C\in D}\bigcup_{e\in C}e\big|-r_M\big(\bigcup_{C\in D}\bigcup_{e\in C}e\big)\Big\}.
\]
Inductively, let $\mathcal{D}_{i+1}=\uparrow\epsilon(\mathcal{D}_i)$ for $i\ge 1$, and $\mathcal{D}=\bigcup_{i\ge 0}\mathcal{D}_i$, where
\[
\epsilon(\mathcal{D}_i):=\mathcal{D}_i\cup\{D_1\cup D_2\setminus C\mid D_1,D_2\in\mathcal{D}_i,\,D_1\cap D_2\notin\mathcal{D}_i,\,C\in D_1\cap D_2\}
\]
and
\[
\uparrow\epsilon(\mathcal{D}_i):=\{D\subseteq\mathcal{C}(M)\mid \exists D'\in\epsilon(\mathcal{D}_i) : D'\subseteq D \}.
\]
The matroid $\delta M:=(\mathcal{C}(M),\mathcal{D})$ with ground set $\mathcal{C}(M)$ is called the {\em combinatorial derived matroid} of $M$, where $\mathcal{D}$ is the collection of all dependent sets of $\delta M$.
}
\end{definition}

Associated with the combinatorial derived matroid, our immediate motive is an attempt to answer the following conjecture.
\begin{conjecture}[\cite{Freij2023},Conjecture 7.6]\label{Conjecture}
Let $M$ be a matroid of rank $r$ and size $m$ such that its dual matroid $M^*$ has an adjoint. Then
$\delta M$ is isomorphic to one of  the adjoints of $M^*$. In particular, the rank of $\delta M$  equals $m-r$.
\end{conjecture}
Inspired by the duality relation in \autoref{Duality},  the standard duality argument motivates another alternative characterization of the adjoint in terms of cocircuits.
\begin{proposition}\label{Adjoint-Cocircuit}
Let $M$, $adM$ be two matroids of the same rank $r$ and $M^*$ be the dual matroid of $M$. Then $adM$ is an adjoint of $M$ if and only if its ground set can be regarded as $E(adM):=\mathcal{C}(M^*)$ such that the sets $C^*[e]:=\{C^*\in\mathcal{C}(M^*)\mid e\notin C^*\}$ are hyperplanes of $adM$ for all $e\in E(M)$ except for loops.
\end{proposition}
\begin{proof}
For the necessity, suppose $adM$ has a ground set $\mathcal{C}(M^*)$ such that the sets $C^*[e]=\{C^*\in\mathcal{C}(M^*)\mid e\notin C^*\}$ are hyperplanes of $adM$ for all $e\in E(M)$ except for loops. Noting that there is a natural one-to-one correspondence between all cocircuits and all hyperplanes of $M$. More precisely, each cocircuit $C^*$ of $M$ corresponds to its hyperplane $H_{C^*}=E(M)\setminus C^*$. Then we can replace every element $C^*$ in $adM$ with $H_{C^*}\in\mathcal{H}(M)$. This yields a matroid $N$ on the ground set $\mathcal{H}(M)$ such that $N\cong adM$, and every set $H[e]=\{H\in\mathcal{H}(M)\mid e\in H\}$ is a hyperplane of $N$ except for loops of $M$. It follows from \autoref{Adjoint-Cha} that $N$ is an adjoint of $M$. So $adM$ is an adjoint of $M$. Similarly, we can verify that the sufficiency  holds.
\end{proof}

Comparing the both different ways to characterize the adjoint in \autoref{Adjoint-Cha} and \autoref{Adjoint-Cocircuit}, it is not difficult to see that the description of an adjoint by the cocircuits in \autoref{Adjoint-Cocircuit} is more close to the combinatorial derived matroid. So, we expect this characterization to be a key bridge between combinatorial derived matroids and adjoints, and even to play an important role in solving those relevant problems in \cite{Freij2023}.  For example, the characterization of an adjoint by the cocircuits makes us more easily see this relation that the Oxley-Wang derived matroid of an $\mathbb{F}$-represented matroid $(M,\varphi)$ is always its an adjoint. Given  a basis $B$ of a matroid $M$ and an element  $e\in E(M)\setminus B$, the unique circuit $C(e;B)$ contained in $B\cup\{e\}$ is called the {\em fundamental circuit} of $e$ with respect to $B$. If $e\in B$, denote $C_{M^*}(e;E(M)\setminus B)$ by $C^*(e;B)$, and call it the {\em fundamental cocircuit} of $e$ with respect to $B$.
\begin{example}\label{Example}
{\rm
Let $(M,\varphi)$ be an $\mathbb{F}$-represented matroid. \cite[Lemma 3]{Oxley-Wang2019} stated that given a basis $B$ of $M$, $\{C(e;B)\mid e\in E(M)\setminus B\}$ forms a basis of $\delta_{OW}M$. Based on the property of $\delta_{OW}M$, suppose $e_0$ is not a coloop of $M$, then there is a basis $B_{e_0}$ of $M$ contained no $e_0$. Let $C[e_0]:=\{C\in\mathcal{C}(M)\mid e_0\notin C\}$. Then $C[e_0]$ contains all the fundamental circuits $C(e;B_{e_0})$ with respect to $B_{e_0}$ except for $C(e_0;B_{e_0})$.  It follows from the elementary linear algebra that the subspace ${\rm span}\big\{\mathbf{c}_{C(e;B_{e_0})}\mid e\in E(M)\setminus (B\cup\{e_0\})\big\}$ and the subspace ${\rm span}\big\{\mathbf{c}_C\mid C\in C[e_0]\big\}$ are the same hyperplane in the space ${\rm span}\big\{\mathbf{c}_C\mid C\in\mathcal{C}(M)\big\}$, where $\mathbf{c}_C$ denotes the circuit vector of a circuit $C$. This implies that every set $C[e]:=\{C\in\mathcal{C}(M)\mid e\notin C\}$  is a hyperplane of $\delta_{OW}M$ except that $e$ is a coloop of $M$. So $\delta_{OW}M$ is an adjoint of the dual matroid $M^*$ by \autoref{Adjoint-Cocircuit}. Namely, $\delta_{OW}M$ is isomorphic to the type I adjoint $\sigma M^*$ of $M^*$. From this perspective, the duality relation in \autoref{Duality} is a special case of the two equivalent characterizations of the adjoint, that one is from the hyperplanes and the other is from the cocircuits.
}
\end{example}

More generally, for an arbitrary matroid $M$, if $M$ has an adjoint $adM$, then the fundamental cocircuits of $M$ also meet the above property as the Oxely-Wang derived matroid.
\begin{lemma}\label{Fundamental-Cocircuit}
Let $M$ be a matroid of rank $r$ and $ad M$ its an adjoint with ground set $\mathcal{C}(M^*)$. If $B$ is a basis of $M$, then $\{C^*(e;B)\mid e\in B\}$ is a basis of $ad M$.
\end{lemma}
\begin{proof}
Let $B=\{e_1,e_2,\ldots,e_r\}$ be a basis of $M$. Since $r(ad M)=r$, this proof reduces to showing that the set $\{C^*(e_i;B)\mid e_i\in B\}$ is independent. Otherwise, this set contains a circuit $C$ of $ad M$. We may assume $C^*(e_r;B)\in C$. On the other hand, the property of $C^*[e_r]$ implies that it contains all the fundamental cocircuits $C^*(e_i;B)$  except for $i=r$. Then we obtain $C^*(e_r;B)\in C^*[e_r]$, which is in contradiction with the fact $e_r\in C^*(e_r;B)$.
\end{proof}

The property in \autoref{Fundamental-Cocircuit} may be an important feature to distinguish the combinatorial derived matroid $\delta M$ of some matroid $M$ and the adjoint of its dual matroid $M^*$. At least this is a necessary condition such that  $\delta M$ is isomorphic to an adjoint of $M^*$. Based on the preceding arguments,  although we could not prove and disprove \autoref{Conjecture}, we provide the following problems which may be  a more specific solution towards \autoref{Conjecture}.
\begin{problem}
Let $M$ be a matroid and $B$ a basis of $M$. If $M^*$ has an adjoint, then the set $\{C(e;B)\mid e\in E(M)\setminus B\}$ is a basis of the combinatorial derived matroid $\delta M$.
\end{problem}
\begin{problem}
Let $M$ be a matroid. If every set $\{C(e;B)\mid e\in E(M)\setminus B\}$ is a basis of $\delta M$ for any basis $B$ of $M$, then $\delta M$ is an adjoint of the dual matroid $M^*$ of $M$.
\end{problem}

In end of this section, we will list a noticeable question for further research.
\begin{question}
{\rm
How to construct a combinatorial adjoint for an arbitrary matroid.
}
\end{question}

\end{document}